\documentclass[12pt, reqno]{amsart}
\usepackage{amsmath, amsthm, amscd, amsfonts, amssymb, graphicx, color}
\usepackage[english]{babel}
\usepackage{cite}
\usepackage[bookmarksnumbered, colorlinks, plainpages]{hyperref}

\newtheorem{theorem}{Theorem}[section]
\newtheorem{lemma}[theorem]{Lemma}

\theoremstyle{definition}

\theoremstyle{remark}
\newtheorem{remark}[theorem]{Remark}
\numberwithin{equation}{section}

\begin{document}

\title{2-Local   derivations
and automorphisms on $B(H)$}


\author{Sh. A. Ayupov}
\address{Institute of
 Mathematics and Information  Technologies,
 Uzbekistan Academy of Sciences,
 100125  Tashkent,   Uzbekistan
 and
 the Abdus Salam International Centre
 for Theoretical Physics (ICTP),
  Trieste, Italy}
\email{sh$_{-}$ayupov@mail.ru}
\thanks{The authors would like to acknowledge
the hospitality of the "Institut f\"{u}r Angewandte Mathematik",
Universit\"{a}t Bonn (Germany). This work is supported in part by
the DFG AL 214/36-1 project (Germany). }

\author{K. K. Kudaybergenov}
\address{Department of Mathematics, Karakalpak state university\\
Ch. Abdirov 1,  230113, Nukus,    Uzbekistan}
\email{karim2006@mail.ru}
\thanks{The second author would also like to acknowledge
 the support of the German Academic Exchange Service -- DAAD}

%    General info

%\date{January 1, 2001 and, in revised form, June 22, 2001.}



\begin{abstract}
The paper is devoted to $2$-local derivations and $2$-local automorphisms on
 the  algebra $B(H)$ of all bounded linear  operators on a
 Hilbert space $H.$ We prove that every  $2$-local derivation on
$B(H)$ is a derivation.
A similar result is obtained for automorphisms.
\end{abstract}

\maketitle
\section{Introduction}

Given an algebra $\mathcal{A},$ a linear operator $D:\mathcal{A}\rightarrow \mathcal{A}$ is
called a \textit{derivation}, if $D(xy)=D(x)y+xD(y)$ for all $x,
y\in \mathcal{A}$ (the Leibniz rule). Each element $a\in \mathcal{A}$ implements a
derivation $D_a$ on $\mathcal{A}$ defined as $D_a(x)=[a, x]=ax-xa,$ $x\in \mathcal{A}.$ Such
derivations $D_a$ are said to be \textit{inner derivations}. If
the element $a,$ implementing the derivation $D_a,$ belongs to a
larger algebra $\mathcal{B}$ containing $\mathcal{A},$ then $D_a$ is called \textit{a
spatial derivation} on $\mathcal{A}.$

There exist various types of linear operators which are close to derivations
\cite{Kad,  Kim,  Sem1}. In particular
R.~Kadison \cite{Kad} has introduced and
investigated so-called local derivations on von Neumann algebras and some polynomial algebras.

A linear operator $\Delta$ on an algebra $\mathcal{A}$ is called a
\textit{local derivation} if given any $x\in \mathcal{A}$ there exists a
derivation $D$ (depending on $x$) such that $\Delta(x)=D(x).$  The
main problems concerning this notion are to find conditions under
which local derivations become derivations and to present examples of algebras
with local derivations that are not derivations \cite{Kad}. In particular
Kadison \cite{Kad} has proved that each
continuous local derivation from a von Neumann algebra
$M$ into a dual $M$-bimodule is a derivation.

In 1997, P. Semrl \cite{Sem1}  introduced the concepts of
$2$-local derivations and $2$-local automorphisms.
A  map $\Delta:\mathcal{A}\rightarrow\mathcal{A}$  (not linear in general) is called a
 $2$-\emph{local derivation} if  for every $x, y\in \mathcal{A},$  there exists
 a derivation $D_{x, y}:\mathcal{A}\rightarrow\mathcal{A}$
such that $\Delta(x)=D_{x, y}(x)$  and $\Delta(y)=D_{x, y}(y).$
A  map $\Theta:\mathcal{A}\rightarrow\mathcal{A}$  (not linear in general) is called a
 $2$-\emph{local automorphism}
 if  for every $x, y\in \mathcal{A},$  there exists
 an automorphism  $\Phi_{x, y}:\mathcal{A}\rightarrow\mathcal{A}$
such that $\Theta(x)=\Phi_{x, y}(x)$  and $\Theta(y)=\Phi_{x, y}(y).$
Local and $2$-local maps have been studied on different operator algebras by many
authors
\cite{Nur,   Bre1, AKNA, Liu,  Kad, Kim, Lar, Lin, Mol, Sem1, Sem2}.

In \cite{Sem1}, P. Semrl described
$2$-local derivations and automorphisms on the algebra $B(H)$ of
 all bounded linear operators on the infinite-dimensional
separable Hilbert space $H.$  A similar
description for the finite-dimensional case appeared later in \cite{Kim},
\cite{Mol}. In the paper \cite{Lin}
$2$-local derivations and automorphisms have been described on
matrix algebras over finite-dimensional division rings.

In the present paper we suggest a new technique and generalize
the above mentioned results of \cite{Sem1} and \cite{Kim} for arbitrary
Hilbert spaces. Namely we consider
$2$-local derivations and $2$-local automorphisms on  the algebra
$B(H)$ of all linear bounded operators on an
arbitrary  (no separability  is assumed)
Hilbert space $H.$ We prove that every $2$-local derivation on $B(H)$
 is a  derivation.
A similar result is obtained for automorphisms, strengthening
 a result of\cite{Liu} in the case of Hilbert spaces. 

\section{Main results}

Let $H$ be an arbitrary Hilbert space, and let
$B(H)$ be the algebra of all linear bounded operators on $H.$
Denote by $\mathcal{F}(H)$  the ideal of all finite-dimensional operators
from $B(H)$ and by $\mbox{tr}$ the canonical trace on $B(H).$

Note that the algebra $\mathcal{F}(H)$
 is semi-prime, i.e. if  $a\in \mathcal{F}(H)$ and
$a \mathcal{F}(H) a=\{0\}$
then $a=0.$ Indeed, let $a\in \mathcal{F}(H)$
 and $a \mathcal{F}(H) a=\{0\},$
i.e. $axa=0$ for all $x\in \mathcal{F}(H).$ In particular for
$x=a^{\ast}$ we have $aa^{\ast}a=0$ and hence $a^{\ast}aa^{\ast}a=0,$ i.e.
$|a|^4=0.$ Therefore $a=0.$

Further any derivation $D$ on $B(H)$ maps the ideal
$\mathcal{F}(H)$ into itself. Indeed, for any
$x\in \mathcal{F}(H)$ there exists a projection $p\in \mathcal{F}(H)$
such that $x=xp.$ Then
$$
D(x)=D(xp)=D(x)p+xD(p),
$$
and therefore $D(x)\in \mathcal{F}(H).$
Hence any $2$-local derivation on $B(H)$ also maps
$\mathcal{F}(H)$ into itself.
Similarly every automorphism on $B(H)$
also  maps $\mathcal{F}(H)$ into itself.

\begin{lemma}\label{A} Let
 $b\in B(H)$ be an arbitrary element.
If  $\mbox{tr}(x b)=0$ for all $x\in \mathcal{F}(H)$
then  $b=0.$
\end{lemma}

\begin{proof} Let  $b\in B(H).$ For any
finite-dimensional projection  $e\in B(H)$ we have $e b^\ast\in
\mathcal{F}(H)$ and therefore
by the assumption of the lemma it follows that
$\mbox{tr}(e b^\ast b)=0.$
Thus
$$
0=\mbox{tr}(e b^\ast b)=\mbox{tr}(e^2 b^\ast b)=
\mbox{tr}(e b^\ast b e)=\mbox{tr}((be)^\ast(be)),
$$
i.e.
$$
\mbox{tr}((be)^\ast(be))=0.
$$
Since the trace $\mbox{tr}$ is faithful, we obtain $(be)^\ast(be)=0,$ i.e. $be=0.$

Now  take a family  of mutually orthogonal  one-dimensional  projections
$\{e_\alpha\}_{\alpha\in J}$ in
 $B(H)$ such that $\bigvee\limits_{\alpha\in J}e_\alpha=\mathbf{1}.$
Given a finite subset $F\subset J$ put $e_F=\sum\limits_{\alpha\in F}e_\alpha.$
We obtain an increasing net $\{e_F\}$, when $F$ runs over all finite subsets of $J$.
Since $e_F\uparrow \mathbf{1}$
we have that
$$
0=be_F b^\ast\uparrow b b^\ast,
$$
i.e. $bb^\ast=0.$ Thus $b=0.$
The proof is complete.
\end{proof}

\begin{lemma}\label{B}
If   $\Delta: B(H)\rightarrow B(H)$ is a $2$-local derivation such that
 $\Delta|_{\mathcal{F}(H)}\equiv 0,$ then $\Delta\equiv 0.$
 \end{lemma}

\begin{proof}
Let $\Delta: B(H)\rightarrow B(H)$ be a $2$-local derivation such that
 $\Delta|_{\mathcal{F}(H)}\equiv 0.$
For arbitrary $x\in B(H)$ and $y\in \mathcal{F}(H)$
there exists a derivation $D_{x, y}$ on $B(H)$
 such that
$\Delta(x)=D_{x, y}(x)$  and $\Delta(y)=D_{x, y}(y).$
By \cite[Corolarry  3.4]{Che} there exists element $a\in B(H)$
such that
$$
[a, xy]=D_{x, y}(xy)=D_{x, y}(x)y+xD_{x, y}(y)=\Delta(x)y+x\Delta(y),
$$
i.e.
$$
[a, xy]=\Delta(x)y+x\Delta(y).
$$
Since $y\in \mathcal{F}(H)$ we have $\Delta(y)=0,$ and therefore
$
[a, xy]=\Delta(x)y.
$
Since  the trace $\mbox{tr}$ accepts finite values on
$\mathcal{F}(H)$ and  $\mathcal{F}(H)$ is an ideal in
$B(H)$  we have
$$
\mbox{tr}(axy) = \mbox{tr}((ax)y) =\mbox{tr}(y(ax)) =
\mbox{tr}((ya)x) = \mbox{tr}(x(ya))= \mbox{tr}(xya).
$$
Thus
$$
0 = \mbox{tr}(axy-xya)=\mbox{tr}([a, xy])=\mbox{tr}\left(\Delta(x)y\right),
$$
i.e.
$\mbox{tr}(\Delta(x)y)=0$ for all $y\in \mathcal{F}(H).$ By Lemma \ref{A}
we have that $\Delta(x)=0.$
The proof is complete. \end{proof}

The following theorem is the main result of this paper.

\begin{theorem}\label{Main}
Let $H$ be an arbitrary Hilbert space, and let
$B(H)$ be the algebra of all bounded linear  operators on $H.$
 Then every $2$-local
derivation $\Delta: B(H)\rightarrow B(H)$
 is a
derivation.
\end{theorem}

\begin{proof} Let $\Delta: B(H)\rightarrow B(H)$ be
 a $2$-local derivation.
For each $x, y\in \mathcal{F}(H)$ there exists a derivation $D_{x, y}$ on $B(H)$
 such that
$\Delta(x)=D_{x, y}(x)$  and $\Delta(y)=D_{x, y}(y).$
By \cite[Corolarry  3.4]{Che} there exists an  element $a\in B(H)$
such that
$$
[a, xy]=D_{x, y}(xy)=D_{x, y}(x)y+xD_{x, y}(y)=\Delta(x)y+x\Delta(y),
$$
i.e.
$$
[a, xy]=\Delta(x)y+x\Delta(y).
$$
Similarly as in Lemma \ref{B}  we have
$$
0 = \mbox{tr}(axy-xya)=
\mbox{tr}([a, xy])=\mbox{tr}\left(\Delta(x)y+x\Delta(y)\right),
$$
i.e.
$\mbox{tr}(\Delta(x)y)=-\mbox{tr}(x \Delta(y)).$
For arbitrary $u, v, w\in \mathcal{F}(H),$
 set $x=u+v,$ $y=w.$ Then from above we obtain
$$
\mbox{tr}(\Delta(u+v)w)=-\mbox{tr}((u+v)\Delta(w))=
$$
$$
=-\mbox{tr}(u\Delta(w))- \mbox{tr}(v\Delta(w))=
\mbox{tr}(\Delta(u)w)+\mbox{tr}(\Delta(v)w)=\mbox{tr}((\Delta(u)+\Delta(v))w),
$$
and so
$$
\mbox{tr}((\Delta(u+v)-\Delta(u)-\Delta(v))w)=0
$$ for all $u, v, w\in \mathcal{F}(H).$
Denote
$
b=\Delta(u+v)-\Delta(u)-\Delta(v)
$ and put
$
w=b^\ast.$ Then
$\mbox{tr}(bb^\ast)=0.$ Since the trace $\mbox{tr}$ is faithful it folows that  $bb^\ast=0,$ i.e. $b=0.$
Therefore
$$
\Delta(u+v)=\Delta(u)+\Delta(v),
$$
i.e. $\Delta$ is an additive map on $\mathcal{F}(H).$

Now let us  show that $\Delta$
 is  homogeneous. Indeed, for
 each $x\in B(H),$ and for
 $\lambda\in \mathbf{C}$ there exists a derivation $D_{x, \lambda x}$
 such that $\Delta(x)=D_{x, \lambda x}(x)$ and
 $\Delta(\lambda x)=D_{x, \lambda x}(\lambda x).$ Then
$$
\Delta(\lambda x)=D_{x, \lambda x}(\lambda x)=\lambda
D_{x, \lambda x}(x)=\lambda\Delta(x).
$$
Hence, $\Delta$ is homogenous and therefore it is a linear operator.

Finally, for each $x\in B(H),$ there exists a derivation $D_{x, x^2}$
 such that $\Delta(x)=D_{x, x^2}(x)$ and $\Delta(x^2)=D_{x, x^2}(x^2).$ Then
$$
\Delta(x^2)=D_{x, x^2}(x^2)=D_{x, x^2}(x)x+xD_{x, x^2}(x)=\Delta(x)x+x\Delta(x)
$$
 for all $x\in B(H).$
Therefore, the restriction $\Delta|_{\mathcal{F}(H)}$ of the operator $\Delta$
on $\mathcal{F}(H)$ is a linear Jordan derivation on
$\mathcal{F}(H)$ in the sense of \cite{Bre2}. In
\cite[Theorem 1]{Bre2} it is proved that any Jordan derivation on
a semi-prime algebra is a derivation. Since $\mathcal{F}(H)$ is semiprime,
therefore the linear operator $\Delta|_{\mathcal{F}(H)}$ is a derivation on
$\mathcal{F}(H).$

Now  by  \cite[Theorem 3.3]{Che}
the derivation  $\Delta|_{\mathcal{F}(H)} : \mathcal{F}(H)\rightarrow \mathcal{F}(H)$
 is spatial, i.e.
\begin{equation}\label{F}
\Delta(x)=ax-xa, \, x\in \mathcal{F}(H)
\end{equation}
for an appropriate $a\in B(H).$

Let us show that  $\Delta(x)=ax-xa$ for all  $x\in B(H).$
Consider the $2$-local derivation  $\Delta_0=\Delta-D_a.$
 Then from  the equality (\ref{F}) we obtain that  $\Delta_0|_{\mathcal{F}(H)}\equiv 0.$
Now by Lemma~\ref{B} it follows that $\Delta_0\equiv 0.$
 This means that  $\Delta=D_a.$
The proof is complete.
\end{proof}

For $2$-local automorphisms of $B(H)$ we have a similar result.

\begin{theorem}\label{LA}
Every $2$-local
automorphism  $\Theta: B(H)\rightarrow B(H)$
 is an
automorphism.
\end{theorem}

\begin{proof}
Since every automorphism on $B(H)$
 maps $\mathcal{F}(H)$ into itself it follows that
  the restriction $\Theta|_{\mathcal{F}(H)}$
is a  2-local automorphism on $\mathcal{F}(H).$
By \cite[Theorem 2.5]{Liu} $\Theta|_{\mathcal{F}(H)}$
is an  automorphism. Therefore by
 \cite[Theorem  3.1]{Che} there exists an invertible
element $a\in B(H)$
such that
$\Theta(x)=a x a^{-1}$
for all $x\in \mathcal{F}(H).$

Let us show that in fact  $\Theta(x)=axa^{-1}$ for all  $x\in B(H).$
Consider the $2$-local automorphism   $\Phi(x)=a^{-1}\Theta(x) a,\, x\in B(H).$
It is clear that   $\Phi(x)=x$ for all $x\in \mathcal{F}(H).$

Now for each $x\in B(H)$ and $y\in \mathcal{F}(H)$
there exists an automorphism  $\Psi_{x, y}$
 such that
$\Phi(x)=\Psi_{x, y}(x)$  and $\Phi(y)=\Psi_{x, y}(y).$
By \cite[Corolarry  3.2]{Che} there exists an invertible element $b\in B(H)$
such that
$$
b(xy)b^{-1}=\Psi_{x, y}(xy)=\Psi_{x,y}(x)\Psi_{x,y}(y)=\Phi(x)y.
$$
Thus
$$
\mbox{tr}(xy)=\mbox{tr}\left(b(xy)b^{-1}\right)=
\mbox{tr}\left(\Phi(x)y\right),
$$
i.e.
$
\mbox{tr}\left([\Phi(x)-x]y\right)=0
$
for all $y\in \mathcal{F}(H).$ By Lemma \ref{A}
we have that $\Phi(x)=x.$
 This means that  $\Theta(x)=axa^{-1}$ for all  $x\in B(H).$
The proof is complete.
\end{proof}

 \begin{remark} A similar result for  $2$-local automorphisms on $B(X),$
where $X$ is a locally convex space, has been obtained in
\cite[Corolarry 2.6]{Liu} under the additional assumption
 of continuity of the map with respect to the weak operator topology.
\end{remark}

\end{document}